\documentclass[12pt,reqno]{article}
\usepackage[usenames]{color}
\usepackage[colorlinks=true,linkcolor=webgreen,filecolor=webbrown,citecolor=webgreen]{hyperref}
\usepackage{amssymb}
\usepackage{amsmath}
\usepackage{amsfonts}
\usepackage[T1]{fontenc}
\usepackage{enumerate}
\usepackage{graphicx}
\usepackage{color}

\definecolor{webgreen}{rgb}{0,.5,0}
\definecolor{webbrown}{rgb}{.6,0,0}
\newtheorem{theorem}{Theorem}

\newtheorem{lemma}[theorem]{Lemma}

\newtheorem{proposition}[theorem]{Proposition}

\newenvironment{proof}[1][Proof]{\noindent\textbf{#1.} }{\ \rule{0.5em}{0.5em}}

\allowdisplaybreaks
\input{tcilatex}
\begin{document}

\begin{center}
\vskip1cm

{\LARGE \textbf{A sharpening of a problem on Bernstein polynomials and
convex functions}}

\vspace{2cm}

{\large Ulrich Abel}\\[3mm]
\textit{Fachbereich MND}\\[0pt]
\textit{Technische Hochschule Mittelhessen}\\[0pt]
\textit{Wilhelm-Leuschner-Stra\ss e 13, 61169 Friedberg, }\\[0pt]
\textit{Germany}\\[0pt]
\href{mailto:Ulrich.Abel@mnd.thm.de}{\texttt{Ulrich.Abel@mnd.thm.de}}\\[1cm]
{\large Ioan Ra\c{s}a}\\[3mm]
\textit{Department of Mathematics}\\[0pt]
\textit{Technical University of Cluj-Napoca,}\\[0pt]
\textit{RO-400114 Cluj-Napoca,}\\[0pt]
\textit{Romania}\\[0pt]
\href{mailto:Ioan.Rasa@math.utcluj.ro}{\texttt{Ioan.Rasa@math.utcluj.ro}}
\end{center}

\vspace{2cm}

{\large \textbf{Abstract.}}

\bigskip

We present an elementary proof of a conjecture proposed by I. Ra\c{s}a \cite%
{Rasa-2017-workshop} in 2017 which is an inequality involving Bernstein
basis polynomials and convex functions. It was affirmed in positive by A.
Komisarski and T. Rajba \cite{Komisarski-letter} very recently by the use of
stochastic convex orderings.

\bigskip



\vspace{2cm}


\section{Introduction}

\label{intro} 
The classical Bernstein polynomials, defined for $f\in C\left[ 0,1\right] $
by 
\begin{equation*}
\left( B_{n}f\right) \left( x\right) =\sum\limits_{\nu =0}^{n}p_{n,\nu
}\left( x\right) {f}\left( \frac{\nu }{n}\right) \text{ \qquad }\left( x\in %
\left[ 0,1\right] \right) ,
\end{equation*}%
with the basis polynomials 
\begin{equation*}
p_{n,\nu }\left( x\right) =\binom{n}{\nu }x^{\nu }\left( 1-x\right) ^{n-\nu }%
\text{ \qquad }\left( \nu =0,1,2,\ldots \right) ,
\end{equation*}%
are the most prominent positive linear approximation operators (see \cite%
{Lorentz-book-1953}). If $f\in C\left[ 0,1\right] $ is convex the inequality 
\begin{equation}
\sum\limits_{i=0}^{n}\sum\limits_{j=0}^{n}\left[ p_{n,i}\left( x\right)
p_{n,j}\left( x\right) +p_{n,i}\left( y\right) p_{n,j}\left( y\right)
-2p_{n,i}\left( x\right) p_{n,j}\left( y\right) \right] {f}\left( \frac{i+j}{%
2n}\right) \geq 0  \label{conjecture-Bernstein}
\end{equation}%
is valid, for $x,y\in \left[ 0,1\right] $.

This inequality involving Bernstein basis polynomials and convex functions
was stated as an open problem 25 years ago by Ioan Ra\c{s}a. During the
Conference on Ulam's Type Stability (Rytro, Poland, 2014), Ra\c{s}a \cite%
{Rasa-2014} recalled his problem.

Inequalities of type $\left( \ref{conjecture-Bernstein}\right) $ have
important applications. They are useful when studying whether the
Bernstein-Schnabl operators preserve convexity (see \cite%
{Altomare-ea-book-2014, altomare-ea-elliptic-,Altomare-ea-Kantorovich-}).

Very recently, J. Mrowiec, T. Rajba and S. W\k{a}sowicz \cite%
{Mrowiec-ea-2017}\ affirmed the conjecture in positive. Their proof makes
heavy use of probability theory. As a tool they applied stochastic convex
orderings (which they proved for binomial distributions) as well as the
so-called concentration inequality. After that one of the authors gave a
short elementary proof \cite{Abel-JAT-2017} of inequality $\left( \ref%
{conjecture-Bernstein}\right) $. The other author remarked in \cite%
{Rasa-2017-workshop} that $\left( \ref{conjecture-Bernstein}\right) $ is
equivalent to 
\begin{equation}
\left( B_{2n}f\right) \left( x\right) +\left( B_{2n}f\right) \left( y\right)
\geq 2\sum\limits_{i=0}^{n}\sum\limits_{j=0}^{n}p_{n,i}\left( x\right)
p_{n,j}\left( y\right) {f}\left( \frac{i+j}{2n}\right) .
\label{conjecture-equivalent}
\end{equation}%
Since $B_{2n}f$ is convex, we have 
\begin{equation}
\left( B_{2n}f\right) \left( x\right) +\left( B_{2n}f\right) \left( y\right)
\geq 2\left( B_{2n}f\right) \left( \frac{x+y}{2}\right) .  \label{Bn-convex}
\end{equation}%
Thus the following problem seems to be a natural one: Prove that 
\begin{equation}
\left( B_{2n}f\right) \left( \frac{x+y}{2}\right) \geq
\sum\limits_{i=0}^{n}\sum\limits_{j=0}^{n}p_{n,i}\left( x\right)
p_{n,j}\left( y\right) {f}\left( \frac{i+j}{2n}\right) ,
\label{new-conjecture}
\end{equation}%
for all convex $f\in C\left[ 0,1\right] $ and $x,y\in \left[ 0,1\right] $.

If $\left( \ref{new-conjecture}\right) $ is valid, then $\left( \ref%
{conjecture-equivalent}\right) $ -- and hence $\left( \ref%
{conjecture-Bernstein}\right) $ -- is a consequence of $\left( \ref%
{Bn-convex}\right) $ and $\left( \ref{new-conjecture}\right) $. Starting
from these remarks, the second author presented the inequality $\left( \ref%
{new-conjecture}\right) $ as an open problem in \cite{Rasa-2017-workshop}. A
probabilistic solution was found by A. Komisarski and T. Rajba \cite%
{Komisarski-letter} using the methods developed in \cite{Mrowiec-ea-2017}
and \cite{Komisarski-2017}.

The purpose of this short note is to give an analytic proof of the following
theorem.

\begin{theorem}
\label{theorem-new-conj-bernstein}Let $n\in \mathbb{N}$. If $f\in C\left[ 0,1%
\right] $ is a convex function, then inequality $\left( \ref{new-conjecture}%
\right) $ is valid for all $x,y\in \left[ 0,1\right] $.
\end{theorem}

\section{An elementary proof of Theorem~\protect\ref%
{theorem-new-conj-bernstein}}

\label{Bernstein} 

The identity 
\begin{equation*}
\sum\limits_{i=0}^{n}\sum\limits_{j=0}^{n}p_{n,i}\left( x\right)
p_{n,j}\left( y\right) {f}\left( \frac{i+j}{2n}\right)
=\sum\limits_{k=0}^{2n}{f}\left( \frac{k}{2n}\right) \frac{1}{k!}\left.
\left( \frac{\partial }{\partial z}\right) ^{k}\left[ \left( 1+xz\right)
^{n}\left( 1+yz\right) ^{n}\right] \right\vert _{z=-1}
\end{equation*}%
is a direct consequence of \cite[Lemma~1]{Abel-JAT-2017}. In particular we
have 
\begin{equation*}
\sum\limits_{i=0}^{n}\sum\limits_{j=0}^{n}p_{n,i}\left( x\right)
p_{n,j}\left( x\right) {f}\left( \frac{i+j}{2n}\right)
=\sum\limits_{k=0}^{2n}{f}\left( \frac{k}{2n}\right) \frac{1}{k!}\left.
\left( \frac{\partial }{\partial z}\right) ^{k}\left[ \left( 1+xz\right)
^{2n}\right] \right\vert _{z=-1}=\left( B_{2n}f\right) \left( x\right) .
\end{equation*}%
Inserting 
\begin{equation*}
\left( B_{2n}f\right) \left( \frac{x+y}{2}\right) =\sum\limits_{k=0}^{2n}{f}%
\left( \frac{k}{2n}\right) \frac{1}{k!}\left. \left( \frac{\partial }{%
\partial z}\right) ^{k}\left( 1+\frac{x+y}{2}z\right) ^{2n}\right\vert
_{z=-1}
\end{equation*}%
we obtain 
\begin{eqnarray*}
&&\left( B_{2n}f\right) \left( \frac{x+y}{2}\right)
-\sum\limits_{i=0}^{n}\sum\limits_{j=0}^{n}p_{n,i}\left( x\right)
p_{n,j}\left( y\right) {f}\left( \frac{i+j}{2n}\right) \\
&=&\sum\limits_{i=0}^{n}\sum\limits_{j=0}^{n}\left[ p_{n,i}\left( \frac{x+y}{%
2}\right) p_{n,j}\left( \frac{x+y}{2}\right) -p_{n,i}\left( x\right)
p_{n,j}\left( y\right) \right] {f}\left( \frac{i+j}{2n}\right) \\
&=&\sum\limits_{k=0}^{2n}{f}\left( \frac{k}{2n}\right) \frac{1}{k!}\left.
\left( \frac{\partial }{\partial z}\right) ^{k}\left[ \left( 1+\frac{x+y}{2}%
z\right) ^{2n}-\left( 1+xz\right) ^{n}\left( 1+yz\right) ^{n}\right]
\right\vert _{z=-1}.
\end{eqnarray*}%
For fixed $n\in \mathbb{N}$ and $x,y\in \left[ 0,1\right] $, we define 
\begin{equation*}
g\left( z\right) \equiv g_{n}\left( z;x,y\right) =z^{-2}\left( \left( 1+%
\frac{x+y}{2}z\right) ^{2n}-\left( 1+xz\right) ^{n}\left( 1+yz\right)
^{n}\right) .
\end{equation*}%
Note that $g$ is a polynomial in $z$ of degree at most $2n-2$.

\begin{lemma}
Fix $x,y\in \left[ 0,1\right] $. Then, the function $g$ satisfies $g^{\left(
k\right) }\left( -1\right) \geq 0$, for $k=0,1,\ldots ,2n-2$.
\end{lemma}

\begin{proof}
Noting that 
\begin{equation*}
\left( 1+\frac{x+y}{2}z\right) ^{2}-\left( 1+xz\right) \left( 1+yz\right)
=\left( \frac{x-y}{2}z\right) ^{2}
\end{equation*}%
we obtain 
\begin{eqnarray*}
g\left( z\right) &=&\frac{1}{4}\left( x-y\right) ^{2}\frac{\left( 1+\frac{x+y%
}{2}z\right) ^{2n}-\left( 1+xz\right) ^{n}\left( 1+yz\right) ^{n}}{\left( 1+%
\frac{x+y}{2}z\right) ^{2}-\left( 1+xz\right) \left( 1+yz\right) } \\
&=&\frac{1}{4}\left( x-y\right) ^{2}\sum\limits_{k=0}^{n-1}\left( 1+\frac{x+y%
}{2}z\right) ^{2\left( n-1-k\right) }\left( 1+xz\right) ^{k}\left(
1+yz\right) ^{k}
\end{eqnarray*}%
it is immediate that $g^{\left( k\right) }\left( -1\right) \geq 0$, for $%
k=0,1,\ldots ,2n-2$, if $x,y\in \left[ 0,1\right] $.
\end{proof}

The key result is the next proposition. The proof follows the lines of \cite[%
Prop.~1]{Abel-JAT-2017}.

\begin{proposition}
\label{prop-bernstein}Fix $x,y\in \left[ 0,1\right] $. Then, for any real
numbers $a_{0},\ldots ,a_{2n}$, the identity 
\begin{equation}
\sum\limits_{i=0}^{n}\sum\limits_{j=0}^{n}\left[ p_{n,i}\left( \frac{x+y}{2}%
\right) p_{n,j}\left( \frac{x+y}{2}\right) -p_{n,i}\left( x\right)
p_{n,j}\left( y\right) \right] \cdot a_{i+j}=\sum\limits_{k=0}^{2n-2}\left(
\Delta ^{2}a_{k}\right) \frac{1}{k!}g^{\left( k\right) }\left( -1\right)
\label{bernstein-identity}
\end{equation}%
is valid.
\end{proposition}

Here $\Delta $ denotes the forward difference $\Delta a_{k}:=a_{k+1}-a_{k}$
such that $\Delta ^{2}a_{k}=a_{k+2}-2a_{k+1}+a_{k}$.

Because $g$ is a polynomial in $z$ of degree at most $2n-2$, it is obvious
that $g^{\left( 2n-1\right) }\left( -1\right) =g^{\left( 2n\right) }\left(
-1\right) =0$.

\begin{proof}[Proof of Theorem~\protect\ref{theorem-new-conj-bernstein}]
For $k=0,1,\ldots ,2n-2$, we put 
\begin{equation*}
a_{k}={f}\left( \frac{k}{2n}\right) .
\end{equation*}%
If $f\in C\left[ 0,1\right] $ is a convex function it follows that $\Delta
^{2}a_{k}\geq 0$, for $k=0,1,\ldots ,2n-2$. Therefore, application of
Proposition~\ref{prop-bernstein} proves Theorem~\ref%
{theorem-new-conj-bernstein}.
\end{proof}

\bigskip


\thispagestyle{empty}

\end{document}